\documentclass[reqno,12pt,onepage]{amsart}
\usepackage{amsthm,amssymb,amsmath, tikz}
\usepackage[pdftex]{hyperref}
\theoremstyle{definition}
\newtheorem{defi}{Definition}[section]
\theoremstyle{remark}
\newtheorem{exam}[defi]{Example}
\theoremstyle{plain}
\newtheorem{thm}[defi]{Theorem}
\newtheorem{lem}[defi]{Lemma}

\newtheorem{prop}[defi]{Proposition}
\newtheorem{rem}[defi]{Remark}

\begin{document}

\title{Local Homeomorphisms that $*$-commute with the Shift}
\author{Paulette N. Willis}

\address{Department of Mathematics \\ University of Houston \\ Houston, TX 77204-3008 \\ USA}
\email{pnwillis@math.uh.edu}

\thanks{This research was partially supported by the University of Iowa Graduate College Fellowship as
part of the Sloan Foundation Graduate Scholarship Program and the University of Iowa Department of Mathematics NSF VIGRE grant DMS-0602242}

\date{\today}



\begin{abstract}
Exel and Renault proved that a sliding block code on a one-sided shift space coming from a progressive block map is a local homeomorphism.  We provide a counterexample showing that the converse does not hold.  We use this example to generalize the notion of progressive to a property of block maps we call \emph{weakly progressive}, and we prove that a sliding block code coming from a weakly progressive block map is a local homeomorphism.  We also introduce the notion of a \emph{regressive} block map and prove that a sliding block code $*$-commutes with the shift map if and only if it comes from a regressive block map.  We also prove that a sliding block code is a local homeomorphism and $*$-commutes with the shift map if and only if it is a $k$-fold covering map defined from a regressive block map.
\end{abstract}

\date{\today}

\numberwithin{equation}{section}
\maketitle

\section{Introduction}

In symbolic dynamics one considers spaces of sequences with entries from a finite alphabet $A$ together with a shift map on the space.  There are two versions of this theory, one that considers the space of two-sided infinite sequences $A^\mathbb{Z}$ with a shift map $\sigma$ that is a homeomorphism, and one that considers the space of one-sided infinite sequences $A^\mathbb{N}$ with a shift map $\sigma$ that is a local homeomorphism.  

Morphisms between shift spaces are called sliding block codes, and one can prove that any such morphism $\tau_d$ comes from a block map $d : A^n \to A$  (see \cite{Hed69, LM95, bK98}).  In \cite[Theorem 3.4]{Hed69} Hedlund proved that the morphisms on two-sided infinite sequences are precisely functions of the form $\sigma^k \tau_d$ for some $k \in \mathbb{Z}$.  In Section~\ref{CF} of this paper we characterize morphisms on one-sided infinite sequences as functions of the form $\tau_d$ (Theorem~\ref{classify}).  More specifically, the morphisms on one-sided infinite sequences are functions $\tau_d: A^{\mathbb{N}} \to A^{\mathbb{N}}$ defined from a block map $d: A^n \to A$, for some $n \in \mathbb{N}$, by $(\tau_d(x))_i=d(x_i \cdots x_{i+n-1})$.  While this result has been stated on several occasions in the literature we include a proof for completeness.

The first main topic of this paper is to consider when a sliding block code on the one-sided shift space is a local homeomorphism.  To date, there is no known characterization.  In \cite[Theorem 14.3]{ER07} Exel and Renault proved that if the block map $d$ that defines $\tau_d$ is a \emph{progressive} function (Definition~\ref{prog}), then $\tau_d$ is a local homeomorphism.  At the time, it was not known if the converse is true. In Section~\ref{LH} we provide a counterexample to the converse (Example~\ref{counterex}).  We use this counterexample as motivation to generalize the idea of a progressive block map and introduce what we call a \emph{weakly progressive} block map (Definition~\ref{weakprog}).  We prove that if the block map $d$ is weakly progressive then $\tau_d$ is a local homeomorphism (Theorem~\ref{wplocalhomeo}).  This gives weaker hypothesis under which we can conclude that $\tau_d$ is a local homeomorphism.  We do not yet know if the converse is true; that is, we do not know if a sliding block code that is a local homeomorphism must come from a weakly progressive block map.

The second main topic of this paper is to examine sliding block codes that $*$-commute with the shift.  The concept of two functions $*$-commuting was introduced in \cite{AR} and further examined in \cite{ER07}.  Let $X$ be a topological space.  Then two functions $S, T:X \to X$ \emph{$*$-commute} if they commute and for all $y,z \in X$ with $S(y)=T(z)$ there exists a unique $x \in X$ such that $T(x)=y$ and $S(x)=z$.  In Section~\ref{SC} we introduce the concept of a \emph{regressive} block map (Definition~\ref {defi-reg}) and characterize the sliding block codes that $*$-commute with the shift as those for which the block map $d$ is regressive.

Local homeomorphisms that $*$-commute with the shift have interesting properties that we discuss in Section~\ref{AI}.  In particular, we prove that a sliding block code is a local homeomorphism and $*$-commutes with the shift if and only if $\phi: A^{\mathbb{N}} \to A^{\mathbb{N}}$ is a $k$-fold covering map coming from a regressive block map (Theorem~\ref{last}).  

The author thanks Ruy Exel for many enlightening discussions on the material.

\textbf{Notation and conventions:}  Throughtout this paper we let $A$ be a finite alphabet. Give $A$ the discrete topology, then $A$ is a compact Hausdorff space.  Let $A^n$ denote the words of length $n$, let $A^* :=\bigcup_{n \geq 1} A^n$, and let $A^{\mathbb{N}}$ denote the one-sided infinite sequence space of elements in $A$.  Since $A$ is a compact Hausdorff space, $A^{\mathbb{N}}$ with the product topology is also a compact Hausdorff space by Tychonoff's Theorem.  For $\mu \in A^*$ we define a \emph{cylinder set} $Z(\mu) := \{ x \in A^{\mathbb{N}}: x_1 \cdots x_{|\mu|}= \mu \}$.  Note that the family $\{Z(\mu): \mu \in A^* \}$ is a basis for $A^{\mathbb{N}}$ and each $Z(\mu)$ is clopen (and therefore compact).  Let $\sigma:A^{\mathbb{N}} \to A^{\mathbb{N}}$ defined by $\sigma(x_1 x_2 x_3 \cdots)= x_2 x_3 \cdots$ be the shift map.

\section{Continuous functions that commute with the shift} \label{CF}

In this section we provide a proof of the fact that a function on the one-sided shift space that is continuous and commutes with the shift is a sliding block code.  Our proof is constructive in the sense that given an arbitrary continuous function $\phi: A^{\mathbb{N}} \to A^{\mathbb{N}}$ that commutes with the shift map $\sigma$, Lemma~\ref{phicont} and Proposition~\ref{prop-d-exists} demonstrate how to construct a block map that defines $\phi$.

\begin{defi} \label{defitaud}
A \emph{block map} is a function $d: A^n \to A$ for some $n \in \mathbb{N}$.  For any block map $d$ we define a function $\tau_d: A^{\mathbb{N}} \to A^{\mathbb{N}}$ by $\tau_d(x)_i= d(x_i \cdots x_{i+n-1})$.  We call $\tau_d$ a \emph{sliding block code}.
\end{defi}

\begin{rem} \label{nonuni}
Define $d: A^2 \to A$ by $d(a_1 a_2)= a_2$.  Then $\tau_d= \sigma$.  Note that $\tau_d$ is not uniquely determined by the choice of $d$.  For example, define $d': A^3 \to A$ by $d'(a_1 a_2 a_3)=a_2$.  Then $\tau_{d'}= \sigma = \tau_d$.
\end{rem}

\begin{lem} \label{taud}
If $\phi: A^{\mathbb{N}} \to A^{\mathbb{N}}$ is a sliding block code, then $\phi$ is continuous and commutes with the shift map $\sigma$.
\end{lem}

\begin{proof}
Since $\phi$ is a sliding block code there exists $n \in \mathbb{N}$ and a block map $d: A^n \to A$ such that $\tau_d=\phi$.  Let $\mu \in A^*$.  If $x= x_1 x_2 x_3 \cdots \in \tau_d^{-1}(Z(\mu))$, then one can show that $x \in Z(x_1 x_2 \cdots x_{k+n-1}) \subseteq \tau_d^{-1}(Z(\mu))$ so $\tau_d^{-1}(Z(\mu))$ is open.  Therefore $\tau_d$ is continuous.  Let $x \in A^{\mathbb{N}}$ and observe
\begin{align*}
\tau_d \sigma (x) &=\tau_d(x_2 x_3 x_4 \cdots)=d(x_2 \cdots x_{n+1}) d(x_3 \cdots x_{n+2}) \cdots \\
                            &=\sigma(d(x_1 \cdots x_n)d(x_2 \cdots x_{n+1}) d(x_3 \cdots x_{n+2}) \cdots) = \sigma \tau_d(x).
\end{align*}
\end{proof}

\begin{rem} \label{samelength}
For each $\mu \in A^*$, $Z(\mu)= \bigsqcup_{a \in A} Z(\mu a)$.  Recall that since $A^{\mathbb{N}}$ is compact, any open set $U \subseteq A^{\mathbb{N}}$ is a finite union of basis elements.  Let $S$ be a finite subset of $A^*$.  Then for any open set $U = \bigcup_{\mu \in S} Z(\mu)$, we may extend the lengths of the $\mu \in S$ so that we may find $n \in \mathbb{N}$ and a finite set $T \subseteq A^n$ such that $\bigcup_{\mu \in S} Z(\mu)= \bigsqcup_{\nu \in T} Z(\nu)$.  Note that for any $\mu \in S$, $| \mu | \leq n$.  Therefore for any $\nu \in T$ there exists $\mu \in S$ such that $Z(\nu) \subseteq Z(\mu)$.
\end{rem}

\begin{lem} \label{phicont}
If $\phi: A^{\mathbb{N}} \to A^{\mathbb{N}}$ is a continuous function, then there exists $n \in \mathbb{N}$ such that for each $\lambda \in A^n$ there exists a unique $a \in A$ such that $Z(\lambda) \subseteq \phi^{-1}(Z(a))$.
\end{lem}

\begin{proof}
Consider $\{Z(a) \, | \, a \in A \}$.  Notice that the $Z(a)$'s are disjoint clopen sets that cover $A^{\mathbb{N}}$.  Define $V_a = \phi^{-1}(Z(a))$.  Then the $V_a$'s are also disjoint clopen sets that cover $A^{\mathbb{N}}$.  Since $A^{\mathbb{N}}$ is compact and the $V_a$'s are closed, the $V_a$'s are also compact.   Therefore each $V_a$ is the union of a finite number of basis elements.  That is, $V_a = \bigcup_{\mu \in S_a} Z(\mu)$, where $S_a$ is a finite subset of $A^*$.  By Remark \ref{samelength} there exists $m \in \mathbb{N}$ and a finite set $T_a \subseteq A^m$ such that $V_a = \bigcup_{\mu \in S_a} Z(\mu)= \bigsqcup_{\nu \in T_a} Z(\nu)$.  Let $T = \bigcup_{a \in A} T_a$.  Then $\bigsqcup_{a \in A} V_a =  \bigsqcup_{\nu \in T} Z(\nu)$ where $T$ is a finite subset of $A^*$.  Observe that the $Z(\nu)$'s are disjoint since the $V_a$'s are disjoint and that for each $\nu \in T$ there exists a unique $a \in A$ such that $Z(\nu) \subseteq V_a$.  By Remark \ref{samelength} there exists $n \in \mathbb{N}$ and a finite set $R \subseteq A^n$ such that $\bigsqcup_{a \in A} V_a =  \bigsqcup_{\nu \in T} Z(\nu)= \bigsqcup_{\lambda \in R} Z(\lambda)$.  Note that since $| \nu | \leq n$ for every $\nu \in T$ and the $Z(\nu)'s$ are disjoint, then for every $\lambda \in R$ there exists a unique $\nu \in T$ such that $Z(\lambda) \subseteq Z(\nu)$.  Recall that $\bigsqcup_{a \in A} V_a$ is a cover of $A^{\mathbb{N}}$.  Therefore $\bigsqcup_{\lambda \in R} Z(\lambda)$ is also a cover of $A^{\mathbb{N}}$, and hence $R=A^n$.  Observe that for each $\lambda \in A^n$, there exists a unique $\nu \in T$ and a unique $a \in A$ such that $Z(\lambda) \subseteq Z(\nu) \subseteq V_a = \phi^{-1}(Z(a))$.
\end{proof}

\begin{prop} \label{prop-d-exists}
If $\phi: A^{\mathbb{N}} \to A^{\mathbb{N}}$ is a continuous function that commutes with the shift map $\sigma$, then $\phi$ is a sliding block code.
\end{prop}

\begin{proof}
Since $\phi$ is continuous, by Lemma \ref{phicont} there exists $n \in \mathbb{N}$ such that for each $\lambda \in A^n$ there exists a unique $a \in A$ such that $Z(\lambda) \subseteq \phi^{-1}(Z(a))$.  Define $d:A^n \to A$ by $d(\lambda)=a$ where $a$ is the unique element in $A$ such that $Z(\lambda) \subseteq \phi^{-1}(Z(a))$.  The function $d$ is well defined since the element $a$ is unique. 
We will now show that $\tau_d = \phi$.  Let $k \in \mathbb{N}$ and $x \in A^{\mathbb{N}}$.  Notice that $\sigma^k(x) \in Z(x_{k+1} \cdots x_{k+n}) \subseteq \phi^{-1}(Z(a))$ for some $a \in A$. Then $\phi(\sigma^k(x))_1=a$.  Also, we have $\tau_d(\sigma^k(x))_1=d(\sigma^k(x)_1 \cdots \sigma^k(x)_n)= d(x_{k+1} \cdots x_{k+n})= a$.  Therefore
\[ \phi(x)_k= \sigma^k(\phi(x))_1= \phi(\sigma^k(x))_1=a =\tau_d(\sigma^k(x))_1 =\sigma^k(\tau_d(x))_1=\tau_d(x)_k.\]
\end{proof}

\begin{exam} \label{constant}
Fix $a \in A$ and define $aaa \cdots= \overline{a} \in A^{\mathbb{N}}$.  Consider a constant function $\phi: A^{\mathbb{N}} \to A^{\mathbb{N}}$ defined by $\phi (y)=\overline{a}$ for all $y \in A^{\mathbb{N}}$.  Then $\phi$ is a continuous function that commutes with $\sigma$.  Notice that for $b \in A$, $\phi^{-1}(Z(b))= \emptyset$ unless $b= a$, thus we have $\phi^{-1}(Z(a))=A$ and $n=1$.  So by Proposition~\ref{prop-d-exists} $d: A \to A$ is defined by $d(b)=a$ for all $b \in A$ and $\tau_d = \phi$.
\end{exam}

\begin{exam}
Let $\phi: A^{\mathbb{N}} \to A^{\mathbb{N}}$ be defined by $\phi(x_1 x_2 x_3 \cdots)=x_2 x_3 \cdots$.  Observe that for $b \in A$, $\phi^{-1}(Z(b))= \bigsqcup_{a \in A} Z(ab)$.  So for all $x_1 x_2 \in A^2$ we have $Z(x_1 x_2) \subset \phi^{-1}(Z(x_2))$.  So by Proposition~\ref{prop-d-exists} $d: A^2 \to A$ is defined by $d(x_1 x_2)=x_2$.
\end{exam}

\noindent The following proposition shows us the extent to which the function $d$ of Proposition~\ref{prop-d-exists} is unique.

\begin{prop} \label{prop-d-uni}
Let $\phi: A^{\mathbb{N}} \to A^{\mathbb{N}}$ be a continuous function that commutes with the shift map $\sigma$.  Let $n$ be the smallest natural number with the property that there exists a block map $d:A^n \to A$ such that $\tau_d=\phi$.  (By Proposition~\ref{prop-d-exists} such an $n$ exists.)  If $m \in \mathbb{N}$ and there exists a function $d': A^m \to A$ such that $\tau_{d'}= \phi$, then $m \geq n$ and $d'(x_i \cdots x_{m+i-1})=d(x_i \cdots x_{n+i-1})$ for all $i \in \mathbb{N}$.  In particular, if $m=n$, then $d'=d$.
\end{prop}

\begin{proof}
Since $\tau_d=\phi=\tau_{d'}$, we have $m \geq n$ by the minimality of $n$.  Also, for all $i \in \mathbb{N}$ and $x \in A^{\mathbb{N}}$ we have $d(x_i \cdots x_{n+i-1})=\tau_d(x)_i= \phi(x)_i= \tau_{d'}(x)_i= d'(x_i \cdots x_{m+i-1})$.
\end{proof}

\begin{thm} \label{classify}
The function $\phi: A^{\mathbb{N}} \to A^{\mathbb{N}}$ is continuous and commutes with the shift map $\sigma$ if and only if $\phi$ is a sliding block code.
\end{thm}

\begin{proof}
The sufficiency is proven in Proposition~\ref{prop-d-exists} and the necessity is proven in Lemma~\ref{taud}.
\end{proof}

The following example illustrates the importance of the function $\phi$ being continuous.

\begin{exam}
Let $A=\{0,1\}$ and $\phi: A^{\mathbb{N}} \to A^{\mathbb{N}}$ be defined as follows: $\phi(0^{\infty})=1^{\infty}$, $\phi(1^{\infty})=0^{\infty}$, and $\phi$ is the identity on all other points.  It is clear that $\phi$ commutes with $\sigma$, however it is impossible to find a function $d$ such that $\tau_d=\phi$.   If $d:A^n \to A$ could be defined for some $n \in \mathbb{N}$, then for the points where $\phi$ acts as the identity we must have that $d(0^n)=0$ and $d(1^n)=1$.  However, defining $d$ in this manner would not work for the points $0^{\infty}$ and $1^{\infty}$.  
\end{exam}

\section{Local homeomorphisms that commute with the shift} \label{LH}

In this section we examine properties on the block map that force the induced sliding block code to be a local homomorphism.  Exel and Renault proved that if the block map is progressive, then the induced sliding block code is a local homeomorphism \cite[Theorem 14.3]{ER07}.  The converse, however, remained an open problem.  In this section we prove the converse is false by providing a counterexample in Example~\ref{counterex}.  Specifically, we describe a sliding block code that is a local homeomorphism such that there does not exist a progressive block map that defines it.   We then generalize the idea of a progressive block map by defining a weakly progressive block map (Definition~\ref{weakprog}).  In Theorem~\ref{wplocalhomeo} we prove that if the block map is a weakly progressive function, then the induced sliding block code is a local homeomorphism.  This gives weaker hypothesis under which we can conclude that a sliding block code is a local homeomorphism.

\begin{defi}
Let $X,Y$ be topological spaces.  A continuous function $f: X \to Y$ is a \emph{local homeomorphism} if for every point $x \in X$ there exists an open neighborhood $U$ of $x$ such that $f(U)$ is open in $Y$ and $f: U \to f(U)$ is a homeomorphism.
\end{defi}

\begin{defi} \label{prog}
A block map $d: A^n \to A$ is \emph{progressive} if for each fixed $x_1 \cdots x_{n-1} \in A^{n-1}$, the function $p_d^{x_1 \cdots x_{n-1}}: A \to A$ defined by $p_d^{x_1 \cdots x_{n-1}}(a)=d(x_1 \cdots x_{n-1}a)$ is bijective .
\end{defi}

\begin{exam} \label{sigma2}
Define $d: A^2 \to A$. by $d(a_1 a_2)= a_2$.  Then $\tau_d=\sigma$.  Fix $a \in A$, let $a_1, a_2 \in A$ and suppose $p_d^a(a_1)= p_d^a(a_2)$.  We have
\[
a_1=d(a a_1)=p_d^a(a_1)=p_d^a(a_2)=d(a a_2)=a_2,
\]
so $p_d^a$ is injective.  Note that $b=p_d^a(b)$ for all $b \in A$ so $p_d^a$ is surjective.  Hence $d$ is progressive.
\end{exam}

\begin{rem}
Recall from Remark~\ref{nonuni} that for $d':A^3 \to A$ defined by $d(a_1 a_2 a_3)=a_2$ we have $\tau_{d'}= \sigma= \tau_d$ for the block map $d$ from Example~\ref{sigma2}.  Notice that $d'$ is not progressive, therefore it is important that we consider the smallest natural number $n$ such that the function $d:A^n \to A$ defines $\tau_d$.

Given an arbitrary sliding block code $\phi$ we wish to determine if there is a progressive block map that defines it.  Let $n$ be the smallest natural number such that a block map $d:A^n \to A$ defines $\phi$.  Then for any $m > n$ and block map $d':A^m \to A$ that defines $\phi$ the function $d'$ is not progressive.  We observe this by recalling from Proposition~\ref{prop-d-uni} that $d'(x_i \cdots x_{m+i-1})= d(x_i \cdots x_{n+i-1})$.  Therefore $d'$ can not be bijective.

In this section, when we consider a block map $d:A^n \to A$ that defines $\tau_d$ we assume that $n$ is the smallest natural number such that there exists a function $d:A^n \to A$ that defines $\tau_d$.  Proposition~\ref{prop-d-uni} allows us to do this.
\end{rem}

\begin{exam}
The constant function $d(b)=a$ for all $b \in A$ from Example~\ref{constant} is not progressive.
\end{exam}

The following is an example of a sliding block code that is a local homeomorphism and can not be defined from a progressive block map.  

\begin{exam} \label{counterex}
Let $A = \{0,1,2,3 \}$ and define $d:A^2 \to A$ by
\begin{align*}
&d(00) = 0  &d(01) &= 0  &d(02) &= 1  &d(03) &= 1\\
&d(10) = 3  &d(11) &= 3  &d(12) &= 2  &d(13) &= 2\\
&d(20) = 2  &d(21) &= 2  &d(22) &= 3  &d(23) &= 3\\
&d(30) = 1  &d(31) &= 1  &d(32) &= 0  &d(33) &= 0.
\end{align*}  
Observe that it is not possible to define a block map $d': A \to A$ such that $\tau_d = \tau_{d'}$.  Therefore if $\tau_d$ may be defined from a progressive block map $d$ is the only possibility.  Since $d(00) = 0= d(01)$ $d$ is not progressive.  With a little work one can check that $\tau_d$ is a homeomorphism on $Z(0) \cup Z(1)$ and $Z(2) \cup Z(3)$ such that $\tau_d(Z(0) \cup Z(1))=A^{\mathbb{N}}= \tau_d(Z(2) \cup Z(3))$.  Therefore $\tau_d$ is a local homeomorphism.
\end{exam}

Now we generalize the idea of a progressive block map by defining a weakly progressive block map.  We prove that if $d$ is weakly progressive, then $\tau_d$ is a local homeomorphism.  

\begin{defi} \label{weakprog}
Fix $n,m \in \mathbb{N}$ and let a block map $d: A^n \to A$ have the property that for every $\mu \in A^n$ and every $\nu \in A^m$ such that $d(\mu)= \nu_1$ there exists a unique $a \in A$ such that $p_{d,m}^{\mu_1 \cdots \mu_{n-1}}(a \alpha)=d(\mu_1 \cdots \mu_{n-1} a) d(\mu_2 \cdots \mu_{n-1} a \alpha_1) \cdots= \nu$ has a solution $\alpha \in A^{m-1}$.  Then we say that $d$ is \emph{weakly progressive} of order $m$.
\end{defi}

\noindent Observe that $d$ is progressive if and only if $d$ is weakly progressive of order $1$.

\begin{exam} \label{wpnotr}
Let $A = \{0,1,2,3 \}$ and define $d:A^2 \to A$ by:
\begin{align*}
&d(00) = 0  &d(01) &= 0 &d(02) &= 1 &d(03) &= 1\\
&d(10) = 2  &d(11) &= 2 &d(12) &= 3 &d(13) &= 3\\
&d(20) = 0  &d(21) &= 0 &d(22) &= 1 &d(23) &= 1\\
&d(30) = 2  &d(31) &= 2 &d(32) &= 3 &d(33) &= 3.
\end{align*}
The function $d$ is weakly progressive of order $2$.
\end{exam}

\begin{rem}
The block map from Example~\ref{counterex} is weakly progressive of order $2$.
\end{rem}

\begin{prop} \label{winimpbij}
Let $d: A^n \to A$ be a block map and fix $x_1 \cdots x_{n-1} \in A^{n-1}$.  If $d$ is weakly progressive, then \[\tau_d: Z(x_1 \cdots x_{n-1}) \to \bigcup_{a \in A} Z(d(x_1 \cdots x_{n-1} a))\] is bijective.
\end{prop}

\begin{proof}
Fix $m$ such that $d$ is weakly progressive of order $m$.  Notice that for $x \in Z(x_1 \cdots x_{n-1})$, $\tau_d(x) \in Z(d(x_1 \cdots x_n)) \subseteq \bigcup_{a \in A} Z(d(x_1 \cdots x_{n-1} a))$.  Therefore we have $\tau_d({Z(x_1 \cdots x_{n-1})}) \subseteq\bigcup_{a \in A} Z(d(x_1 \cdots x_{n-1} a))$.  Let $y \in Z(d(x_1 \cdots x_{n-1} a))$ for some $a \in A$.  We want to show that there exists a unique $x \in Z(x_1 \cdots x_{n-1})$ such that $\tau_d(x)= y$.  Notice that $x_1 \cdots x_{n-1} a \in A^n$ and $y_1 \cdots y_m \in A^m$ satisfy $d(x_1 \cdots x_{n-1}a)= y_1$.  So since $d$ is weakly progressive there exists a unique $a_1 \in A$ such that
\[
p_{d,m}^{x_1 \cdots x_{n-1}}(a_1  \alpha)= d(x_1 \cdots x_{n-1} a_1) d(x_2 \cdots x_{n-1} a_1 \alpha_1) \cdots= y_1 \cdots y_m
\]
for some $\alpha \in A^{m-1}$.  Now consider $x_2 \cdots x_{n-1} a_1 \alpha_1 \in A^n$ and $y_2 \cdots y_{m+1} \in A^m$ such that $d(x_2 \cdots x_{n-1} a_1 \alpha_1)= y_2$.  Since $d$ is weakly progressive there exists a unique $a_2 \in A$ such that
\[
p_{d,m}^{x_1 \cdots x_{n-1}}(a_2 \beta) = d(x_2 \cdots x_{n-1} a_1 a_2) d(x_3 \cdots x_{n-1} a_1 a_2 \beta_1) = y_2 \cdots y_{m+1}
\]
for some $\beta \in A^{m-1}$.  We may continue in this manner to construct $x= x_1 \cdots x_{n-1} a_1 a_2 \cdots$ such that $\tau_d(x)=y$, hence the function is surjective.  Since each $a_i$ was unique we have $\tau_d|{(Z(x_1 \cdots x_{n-1}))}$ injective.
\end{proof}

\noindent Observe that if $d$ is progressive, then $\bigcup_{a \in A} Z(d(x_1 \cdots x_{n-1} a))= A^{\mathbb{N}}$.  Thus $\tau_d$ is $|A^{n-1}|$ to $1$.

\begin{thm} \label{wplocalhomeo}
If $d: A^n \to A$ is weakly progressive block map then the induced sliding block code $\tau_d: A^{\mathbb{N}} \to A^{\mathbb{N}}$ is a local homeomorphism.
\end{thm}

\begin{proof}
By Proposition~\ref{winimpbij} $\tau_d(Z(x_1 \cdots x_{n-1})) = \bigcup_{a \in A} Z(d(x_1 \cdots x_{n-1} a))$ so $\tau_d(Z(x_1 \cdots x_{n-1}))$ is open.  Since $\tau_d$ is continuous, $\tau_d|_{Z(x_1 \cdots x_{n-1})}$ is also continuous.  The set $Z(x_1 \cdots x_{n-1})$ is compact since it is a cylinder set.  Recall $A^{\mathbb{N}}$ is Hausdorff, hence $\bigcup_{a \in A} Z(d(x_1 \cdots x_{n-1} a)) \subseteq A^{\mathbb{N}}$ is Hausdorff.  So we have $\tau_d|_{Z(x_1 \cdots x_{n-1})}$ is a continuous bijective function from the compact space $Z(x_1 \cdots x_{n-1})$ to the Hausdorff space $\bigcup_{a \in A} Z(d(x_1 \cdots x_{n-1} a))$.  Therefore by \cite[Theorem 5.8]{jKel55}, $\tau_d|_{Z(x_1 \cdots x_{n-1})}$ is a homeomorphism.  Hence $\tau_d$ is a local homeomorphism.
\end{proof}

\section{Continuous functions that $*$-commute with the Shift} \label{SC}

The concept of $*$-commuting for functions was introduced in \cite{AR} and further examined in \cite{ER07}.  In this section we introduce the concept of a regressive block map (Definition~\ref {defi-reg}) and prove that sliding block codes that $*$-commute with the shift map $\sigma$ are exactly those defined from regressive block maps.

\begin{defi} \label{Exel}
Let $X$ be a set.  Two functions $S,T:X \to X$ \emph{$*$-commute} if they commute and given $(y,z) \in X \times X$ such that $S(y)=T(z)$ there exists a unique $x \in X$ such that $T(x)=y$ and $S(x)=z$. 

\[\begin{tikzpicture}%[scale=1.5]
    \node[inner sep=0pt, circle, fill=black] (a) at (0, 2) [draw] {.}; 
    \node[inner sep=3pt, anchor = south] at (a.north) {$x$};
    \node[inner sep=0pt, circle, fill=black] (b) at (-2, 0) [draw] {.}; 
    \node[inner sep=3pt, anchor = south] at (b.north) {$y$};
    \node[inner sep=0pt, circle, fill=black] (c) at (0, -2) [draw] {.}; 
    \node[inner sep=3pt, anchor = north] at (c.south) {$S(y)=T(z)$};
    \node[inner sep=0pt, circle, fill=black] (d) at (2, 0) [draw] {.}; 
    \node[inner sep=3pt, anchor = south] at (d.north) {$z$};
     \draw[style=semithick, style=dashed, -latex] (a.south west)--(b.north east) node[pos=0.5, anchor=south, inner sep=4pt]{$T$};
      \draw[style=semithick, -latex] (b.south east)--(c.north west) node[pos=0.5, anchor=south, inner sep=3pt]{$S$};
      \draw[style=semithick, style=dashed, -latex] (a.south east)--(d.north west) node[pos=0.5, anchor=south, inner sep=4pt]{$S$};
      \draw[style=semithick, -latex] (d.south west)--(c.north east) node[pos=0.5, anchor=south, inner sep=4pt]{$T$};
\end{tikzpicture}\]
\end{defi}

\begin{rem}
 Definition~\ref{Exel} is the same as the definition of ``star-commuting" in \cite[Definition 10.1]{ER07}. This is equivalent to the definition of ``$*$-commuting" given in \cite[Definition 5.6]{AR}.
\end{rem}

\begin{exam} \label{bar}
Let $A=\{0,1\}$ and define the function $d:A \to A$ by $d(0)=1$ and $d(1)=0$.  For $x \in A^{\mathbb{N}}$ denote $\tau_d(x)=\overline{x}$.  By Lemma \ref{taud} we know that $\tau_d$ commutes with $\sigma$.  Let $y,z \in A^{\mathbb{N}}$ be such that $\sigma(y)=\tau_d(z)$.  Since $\tau_d$ is bijective, observe that $\overline{y}$ is the unique element in $A^{\mathbb{N}}$ such that $\tau_d(\overline{y})=y$.  We also have that $\tau_d(\sigma(\overline{y}))=\sigma(\tau_d(\overline{y}))= \sigma(y) = \tau_d(z)$ and since $\tau_d$ is bijective $\sigma(\overline{y})=z$.  So $\tau_d$ $*$-commutes with $\sigma$.
\end{exam}

\begin{exam} \label{sigma3}
Recall Example \ref{sigma2} where we define $d: A^2 \to A$ by $d(a_1 a_2)= a_2$ so that $\tau_d= \sigma$.  Let $a_1, a_2  \in A$ such that $a_1 \neq a_2$ and $w \in A^{\mathbb{N}}$.  Observe that $\sigma(a_1 w)=w=\sigma(a_2 w)$ and $a_1 w \neq a_2 w$.  Therefore $\sigma$ does not $*$-commute with itself.
\end{exam}

In the two previous examples proving whether or not the function $\tau_d$ $*$-commutes with $\sigma$ using the definition was not terribly difficult.  However consider the following example: Let $A=\{0,1,\cdots n-1\}$ and define $d:A^n \to A$ by $d(a_1 \cdots a_n)= (a_1+ \cdots + a_n) \pmod{n}$.  Determining whether or not the associated $\tau_d$ $*$-commutes with $\sigma$ is extremely unpleasant.  We would like to determine easily verifiable conditions on the block map $d$ that determine when $\tau_d$ $*$-commutes with the shift.

\begin{defi} \label{defi-reg}
The block map $d: A^n \to A$ is \emph{regressive} if for each fixed $x_1 \cdots x_{n-1} \in A^{n-1}$ the function $r_d^{x_1 \cdots x_{n-1}}: A \to A$ defined by $r_d^{x_1 \cdots x_{n-1}}(a)=d(ax_1 \cdots x_{n-1})$ is bijective.
\end{defi}

\begin{exam} \label{counterex2}
Recall the block map from Example~\ref{counterex}.  Notice that when the second coordinate is fixed $d$ is bijective.  Therefore $d$ is regressive.
\end{exam}

\begin{exam} 
Recall the block map from Example~\ref{wpnotr}.  Since $d(00) = 0= d(20)$, $d$ is not regressive.
\end{exam}

\begin{exam} \label{modn}
For this example, all addition is modulo $n$.  Let $A=\{0,1,\cdots, n-1\}$ and define $d:A^n \to A$ by $d(a_1 \cdots a_n)= (a_1+ \cdots + a_n) \pmod{n}$.  Fix $x_1 \cdots x_{n-1} \in A^{n-1}$ and let $x:= x_1+ \cdots + x_{n-1}$.  To see that $r_d$ is injective, let $a_1, a_2 \in A$ and suppose $r_d^{x_1 \cdots x_{n-1}}(a_1)=r_d^{x_1 \cdots x_{n-1}}(a_2)$.  Then
 \begin{align*}
a_1+x &=d(a_1 x_1 \cdots x_{n-1})= r_d^{x_1 \cdots x_{n-1}}(a_1) \\
           &=r_d^{x_1 \cdots x_{n-1}}(a_2)= d(a_2 x_1 \cdots x_{n-1})=a_2+x,
 \end{align*}
therefore $a_1=a_2$.  Let $a \in A$.  Then we have $r_d^{0_1 \cdots 0_{n-1}}(a)=a$.  Therefore $d$ is regressive.
\end{exam}

\begin{thm} \label{refiffstar}
The block map $d: A^n \to A$ is regressive if and only if the induced sliding block code $\tau_d: A^{\mathbb{N}} \to A^{\mathbb{N}}$ $*$-commutes with the shift map $\sigma$.
\end{thm}

\begin{proof}
By Lemma~\ref{taud}, $\tau_d$ commutes with $\sigma$.  Suppose we have $y,z \in A^{\mathbb{N}}$ such that $\sigma (y) = \tau_d(z)$.  Since $d$ is regressive there exists a unique $x_1 \in A$ such that $r_d^{z_1 \cdots z_{n-1}}(x_1)= d(x_1 z_1 \cdots z_{n-1})=y_1$.  Notice that $y_{i+1}=\sigma(y)_i= \tau_d(z)_i= \tau_d(x_1 z)_{i+1}$.  So we have
\[
\tau_d(x_1 z)= d(x_1 z_1 \cdots z_{n-1}) \tau_d(x_1 z)_2 \tau_d(x_1 z)_3 \cdots= y_1 y_2 y_3 \cdots= y
\]
and $\sigma(x_1 z)=z$.  To see that $x_1z$ is unique suppose there exists $w \in A^{\mathbb{N}}$ such that $\tau_d(w)=y$ and $\sigma(w)= z$.  Then $w= az$ for some $a \in A$.  Notice that $ d(az_1 \cdots z_{n-1})= \tau_d(az)_1 =y_1=d(x_1z_1 \cdots z_{n-1})$.  Since $d$ is regressive $a=x_1$.  Therefore $\tau_d$ $*$-commutes with $\sigma$.

Conversely, fix $x_1 \cdots x_{n-1} \in A^{n-1}$.  Suppose for $a_1, a_2 \in A$ we have $r_d^{x_1 \cdots x_{n-1}}(a_1)=r_d^{x_1 \cdots x_{n-1}}(a_2)$.  Then let $z \in Z(x_1 \cdots x_{n-1})$ and observe that
\begin{align*}
\tau_d(a_1 z)_1 &= d(a_1 z_1, \cdots z_{n-1})= r_d^{x_1 \cdots x_{n-1}}(a_1)\\
                          &=r_d^{x_1 \cdots x_{n-1}}(a_2)= d(a_2 z_1, \cdots z_{n-1})= \tau_d(a_2 z)_1.
\end{align*}
For $i \geq 2$ we have $\tau_d(a_1 z)_i= \tau_d(z)_{i-1}= \tau_d(a_2 z)_i$.  So $\tau_d(a_1 z)= \tau_d(a_2 z)$ and $\sigma(a_1 z)=z= \sigma(a_2 z)$. Since $\tau_d$ $*$-commutes with $\sigma$ we have $a_1 z= a_2 z$.  Therefore $r_d^{x_1 \cdots x_{n-1}}$ is injective.  Now let $a \in A$.  Suppose $z \in Z(x_1 \cdots x_{n-1})$ and define $w= \tau_d(z)$.  Then $aw, z \in A^{\mathbb{N}}$ satisfy $\sigma(aw)= \tau_d(z)$.  Since $\tau_d$ and $\sigma$ $*$-commute there exists a unique $v \in A^{\mathbb{N}}$ such that $\sigma(v)=z$ and $\tau_d(v)=aw$.  Since $\sigma(v)=z$, there exists $b \in A$ such that $v=bz$.  So we have $a= \tau_d(v)_1= \tau_d(bz)_1= d(bz_1 \cdots z_{n-1})= d(bx_1 \cdots x_{n-1})$.  So $b \in A$ such that $r_d^{x_1 \cdots x_{n-1}}(b)= d(bx_1 \cdots x_{n-1})= a$.  Therefore $d$ is regressive.
\end{proof}

\section{Local homeomorphisms that $*$-commute with the shift} \label{AI}

In this section we examine properties of sliding block codes that are local homomorphisms and $*$-commute with the shift.  In Theorem~\ref{last} we show that this class of functions is precisely the $k$-fold covering maps defined from regressive block maps.

\begin{defi}
Let $X$ be a topological space, $\phi: X \to X$ be a function and $k \in \mathbb{N}$.  We define the sets $Z_k^{\phi}:= \{y \in X: |\phi^{-1}(y)| = k \}$ and $Z_{\geq k}^{\phi}:= \{y \in X: |\phi^{-1}(y)| \geq k \}$.
\end{defi}

\begin{rem} \label{notempty}
Let $\phi: A^{\mathbb{N}} \to A^{\mathbb{N}}$ commute with $\sigma$ and fix $y \in A^{\mathbb{N}}$.  If there exists $x \in A^{\mathbb{N}}$ such that $\phi (x)=y$, then $\sigma(y)= \sigma (\phi(x))= \phi(\sigma(x))$.  This shows that if $\phi^{-1}(y) \neq \emptyset$, then $\phi^{-1}(\sigma(y)) \neq \emptyset$.
\end{rem}

\begin{defi}
Let $X \subseteq A^{\mathbb{N}}$.  We say that $X$ is \emph{shift invariant} if $\sigma(X)=X$.
\end{defi}

\begin{prop} \label{preim}
If $\phi: A^{\mathbb{N}} \to A^{\mathbb{N}}$ $*$-commutes with the shift map $\sigma$, then $Z_k^{\phi}$ is shift invariant for all $k \in \mathbb{N}$. 
\end{prop}

\begin{proof}
Let $y \in A^{\mathbb{N}}$ and fix $k \in \mathbb{N}$ such that  $\sigma(y) \in Z_k^{\phi}$.  Then $\sigma(y)$ has $k$ preimages under $\phi$ and we define $\{ z^i\}^k_{i=1} = \phi^{-1}(\sigma(y))$ (see Figure 1 below).  Since $\sigma$ and $\phi$ $*$-commute, for each $z^i$ there exists a unique $x^i$ such that $\phi(x^i)=y$ and $\sigma(x^i)=z^i$.  So $|\phi^{-1}(y)| \geq k$, but we want to show that $|\phi^{-1}(y)|=k$.  Suppose $x \in \phi^{-1}(y)$.  Then $\phi(\sigma(x))=\sigma(\phi(x))=\sigma(y)$.  So $\sigma(x) \in \phi^{-1}(\sigma(y))= \{ z^i\}^k_{i=1}$, and there exists $i$ such that $\sigma(x)=z^i$.  Hence $\phi(x)=y$, $\sigma(x)=z^i$, but $x^i$ is the unique element with those properties thus $x=x^i$.  Therefore $|\phi^{-1}(y)|=k$ that means $y \in Z_k^{\phi}$.

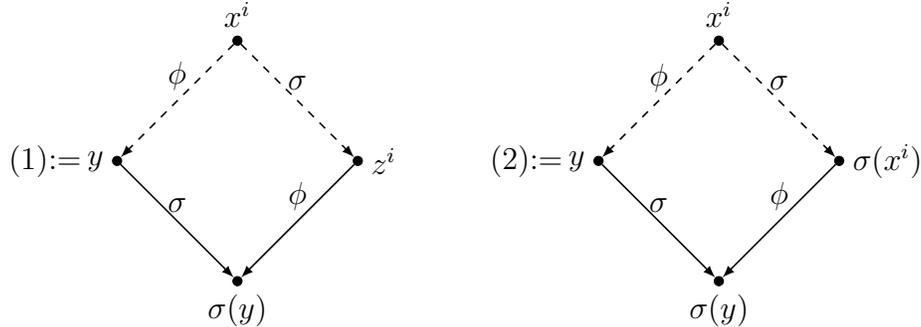
\[\begin{tikzpicture}[scale=0.8]
    \node[inner sep=2pt, anchor = east] at (-6.5,0) {(1):=};
    \node[inner sep=0pt, circle, fill=black] (a) at (-4, 2) [draw] {.}; 
    \node[inner sep=3pt, anchor = south] at (a.north) {$x^i$};
    \node[inner sep=0pt, circle, fill=black] (b) at (-6, 0) [draw] {.}; 
    \node[inner sep=3pt, anchor = east] at (b.west) {$y$};
    \node[inner sep=0pt, circle, fill=black] (c) at (-4, -2) [draw] {.}; 
    \node[inner sep=3pt, anchor = north] at (c.south) {$\sigma(y)$};
    \node[inner sep=0pt, circle, fill=black] (d) at (-2, 0) [draw] {.}; 
    \node[inner sep=3pt, anchor = west] at (d.east) {$z^i$};
     \draw[style=semithick, style=dashed, -latex] (a.south west)--(b.north east) node[pos=0.5, anchor=south, inner sep=4pt]{$\phi$};
      \draw[style=semithick, -latex] (b.south east)--(c.north west) node[pos=0.5, anchor=south, inner sep=3pt]{$\sigma$};
      \draw[style=semithick, style=dashed, -latex] (a.south east)--(d.north west) node[pos=0.5, anchor=south, inner sep=4pt]{$\sigma$};
      \draw[style=semithick, -latex] (d.south west)--(c.north east) node[pos=0.5, anchor=south, inner sep=4pt]{$\phi$};

   \node[inner sep=2pt, anchor = east] at (1.5,0) {(2):=};
    \node[inner sep=0pt, circle, fill=black] (a) at (4, 2) [draw] {.}; 
    \node[inner sep=3pt, anchor = south] at (a.north) {$x^i$};
    \node[inner sep=0pt, circle, fill=black] (b) at (2, 0) [draw] {.}; 
    \node[inner sep=3pt, anchor = east] at (b.west) {$y$};
    \node[inner sep=0pt, circle, fill=black] (c) at (4, -2) [draw] {.}; 
    \node[inner sep=3pt, anchor = north] at (c.south) {$\sigma(y)$};
    \node[inner sep=0pt, circle, fill=black] (d) at (6, 0) [draw] {.}; 
    \node[inner sep=3pt, anchor = west] at (d.east) {$\sigma(x^i)$};
     \draw[style=semithick, style=dashed, -latex] (a.south west)--(b.north east) node[pos=0.5, anchor=south, inner sep=3pt]{$\phi$};
      \draw[style=semithick, -latex] (b.south east)--(c.north west) node[pos=0.5, anchor=south, inner sep=3pt]{$\sigma$};
      \draw[style=semithick, style=dashed, -latex] (a.south east)--(d.north west) node[pos=0.5, anchor=south, inner sep=4pt]{$\sigma$};
      \draw[style=semithick, -latex] (d.south west)--(c.north east) node[pos=0.5, anchor=south, inner sep=4pt]{$\phi$};
\end{tikzpicture}\]

Conversely, let $y \in A^{\mathbb{N}}$ and fix $k \in \mathbb{N}$ such that $y \in Z_k^{\phi}$.  Then define $\{x^i \}^k_{i=1}= \phi^{-1}(y)$ (see Figure 2 above).  Suppose $w \in \phi^{-1}(\sigma(y))$ (which exists by Remark~\ref{notempty}).  Since $\sigma$ and $\phi$ $*$-commute there exists $x$ such that $\phi(x)=y$ and $\sigma(x)=w$.  However $\{x^i \}^k_{i=1}= \phi^{-1}(y)$ so $x = x^i$ for some $i$.  So for each $w \in \phi^{-1}(\sigma(y))$, $w= \sigma(x^i)$ for some $i$.  So $|\phi^{-1}(\sigma(y))| \leq k$.  Suppose $|\phi^{-1}(\sigma(y))| < k$, then there exists $x^i, x^j \in \phi^{-1}(y)$ with $i \neq j$ such that $\sigma(x^i)= \sigma(x^j)= z$, say.  So we have $y,z \in A^{\mathbb{N}}$ such that $\sigma(y)=\phi(z)$ and $x^i \neq x^j$ such that $\phi(x^i)=y=\phi(x^j)$ and $\sigma(x^i)=z=\sigma(x^j)$.  Since $\phi$ and $\sigma$ $*$-commute $x^i=x^j$, that is a contradiction.  Therefore $|\phi^{-1}(\sigma(y))| = k$ which means $\sigma(y) \in Z_k^{\phi}$.
\end{proof}

\begin{prop} \label{open}
If $\phi: A^{\mathbb{N}} \to A^{\mathbb{N}}$ is a local homeomorphism, then $Z^{\phi}_{\geq k}$ is open in $A^{\mathbb{N}}$ for all $k \in \mathbb{N}$.
\end{prop}

\begin{proof}
Let $y \in A^{\mathbb{N}}$ and fix $k \in \mathbb{N}$ such that $y \in Z^{\phi}_{\geq k}$.  Then there exists $l \geq k$ such that  $\phi^{-1}(y)=\{x_i\}^l_{i=1}$ where each $x_i$ is distinct.  Therefore $y \in Z^{\phi}_l$.  Since $\phi$ is a local homeomorphism, for each $x_i$ there exists a neighborhood $U_i$ containing $x_i$ such that $x_j$ is not an element of $U_i$ for $i \neq j$, $y \in \phi(U_i)$ for each $i$, and $\phi(U_i)$ is open in $A^{\mathbb{N}}$.  Since $A^{\mathbb{N}}$ is Hausdorff, let $x_i \in V_i$ for each $i$ and $V_i \cap V_j = \emptyset$ for each $i \neq j$.  Then define $W_i= V_i \cap U_i$.  Thus the set $\{W_i\}^l_{i=1}$ are pairwise disjoint.  Now let $W= \bigcap^l_{i=1} \phi(W_i)$.  Then $W$ is open in $A^{\mathbb{N}}$ and $y \in W$.  We want to show that $W \subseteq Z^{\phi}_{\geq k}$.  Let $z \in W$.  Then $z \in \phi(W_i)$ for each $i$, so there exists $w_i \in W_i$ such that $\phi(w_i)=z$.  Since $\{W_i\}^l_{i=1}$ are pairwise disjoint, each $w_i$ is distinct.  Therefore $z$ has at least $l$ preimages.  Hence $z \in Z^{\phi}_{\geq l} \subseteq Z^{\phi}_{\geq k}$, thus $W \subseteq Z^{\phi}_{\geq k}$.  Therefore $Z^{\phi}_{\geq k}$ is open.
\end{proof}

\begin{rem} \label{wholeX}
It is important to note that the only shift invariant open sets in $A^{\mathbb{N}}$ are $\emptyset$ and $A^{\mathbb{N}}$.  Any non-empty open set $U$ contains a basic open set $Z(x_1, \cdots, x_l)$ for some $l \in \mathbb{N}$.  If $\sigma(U)=U$, then $\sigma^k(U)=U$ for any $k \in \mathbb{N}$.  So $\sigma^l(Z(x_1, \cdots, x_l))=A^{\mathbb{N}} \subseteq U$.  
\end{rem}

\begin{lem} \label{MN}
If $\phi: A^{\mathbb{N}} \to A^{\mathbb{N}}$ is a local homeomorphism, then there exists an $M \in \mathbb{N}$ such that $\{Z(\mu): \mu \in A^M \}$ is a finite covering of $A^{\mathbb{N}}$ by disjoint sets and $\phi$ is a homeomorphism on each $Z(\mu)$.
\end{lem}

\begin{proof}
Let $\{W_{\alpha}\}$ be a covering basic open sets of $A^{\mathbb{N}}$ such that $\phi$ is a homeomorphism on each set and let $\{W_i\}^n_{i=1}$ be a finite subcover.  By Remark~\ref{samelength} there exists $M \in \mathbb{N}$ and $T \subseteq A^M$ such that $\bigcup_{i=1}^n W_i = \bigsqcup_{\nu \in T} Z(\nu)$.  Since $\{W_i\}^n_{i=1}$ is a cover, $T= A^M$.  Observe that for each $\nu \in A^M$, $Z(\nu) \subseteq W_i$, therefore $\phi$ is a homeomorphism on each $Z(\nu)$.
\end{proof}

\begin{prop}\label{ksurj}
If the sliding block code $\phi: A^{\mathbb{N}} \to A^{\mathbb{N}}$ is a local homeomorphism that $*$-commutes with the shift map $\sigma$, then $\phi$ is surjective and there exists $k \in \mathbb{N}$ such that $\phi$ is $k$-to-$1$.
\end{prop}

\begin{proof}
Since $\phi$ is a local homeomorphism, by Lemma~\ref{MN} there exists an $M \in \mathbb{N}$ such that $\phi$ is a homeomorphism on each $Z(\mu)$ for all $\nu \in A^M$.  Observe that for any $l \in \mathbb{N}$ such that $l > |A^M|$ we have $Z_l^{\phi}= \emptyset$.  Let $k := \operatorname{max} \{l: Z_l^{\phi} \neq \emptyset \}$ and notice that $Z_k^{\phi} =Z_{\geq k}^{\phi}$.  Since $\phi$ $*$-commutes with $\sigma$, by Proposition~\ref{preim} we have $\sigma(Z_k^{\phi}) = Z_k^{\phi}$.  By Proposition~\ref{open} $Z_k^{\phi}=Z_{\geq k}^{\phi}$ is open and by Remark~\ref{wholeX} $Z_k^{\phi} =A^{\mathbb{N}}$.  Thus every element of $A^{\mathbb{N}}$ has exactly $k$ preimages under $\phi$.

Suppose $\phi$ is not surjective.  Then there exists $y \in A^{\mathbb{N}}$ such that $y \in Z_0$.  We have just proven that every element of $A^{\mathbb{N}}$ must have the same number of preimages under $\phi$.  Then $k=0$ which means $\phi$ is not defined for any element in $A^{\mathbb{N}}$.  Thus $\phi$ must be surjective.
\end{proof}

\begin{defi}
Let $p: E \to B$ be a continuous surjective function.  The open set $U$ of $B$ is said to be \emph{evenly covered} by $p$ if the inverse image $p^{-1}(U)$ can be written as the union of disjoint open sets $V_{\alpha}$ in $E$ such that for each $\alpha$, the restriction of $p$ to $V_{\alpha}$ is a homeomorphism of $V_{\alpha}$ onto $U$.
\end{defi}

\begin{defi}
Let $p: E \to B$ be a continuous surjective function.  If every point $b \in B$ has a neighborhood $U$ that is evenly covered by $p$, then $p$ is called a \emph{covering map} and $E$ is said to be a \emph{covering space} of $B$.  If $p^{-1}(b)$ has $k$ elements for every $b \in B$, then $E$ is called a \emph{$k$-fold covering} of $B$.  The condition that $p$ be a local homeomorphism does not suffice to ensure that $p$ is a covering map (see \cite[Chapter 9, page 338, Example 2]{Mun}).
\end{defi}

\begin{exam}
The shift $\sigma$ is a $|A|$-fold covering map.
\end{exam}

\begin{exam}
Define $V:= Z(0) \cup Z(1)$ and $W := Z(2) \cup Z(3)$.  Observe that for both the $\tau_d$ functions from Example~\ref{counterex} and Example~\ref{wpnotr} we have $\tau_d(V) = A^{\mathbb{N}}=\tau_d(W)$.  Therefore the $\tau_d$ functions are $2$-fold covering maps.
\end{exam}

\begin{prop} \label{kcm}
If the sliding block code $\phi: A^{\mathbb{N}} \to A^{\mathbb{N}}$ is a local homeomorphism that $*$-commutes with the shift map $\sigma$, then $\phi$ is a $k$-fold covering map.
\end{prop}

\begin{proof}
The function $\phi$ is continuous by definition and surjective by Proposition~\ref{ksurj}.  Let $y \in A^{\mathbb{N}}$, then $\phi^{-1}(y)= \{ x_i \}_{i=1}^k$ for some $k \in \mathbb{N}$ by Proposition~\ref{ksurj}.  Since $\phi$ is  a local homeomorphism there exists an open neighborhood $W_i$ of $x_i$ such that $\phi: W_i \to \phi(W_i)$ is a homeomorphism and $\phi(W_i)$ for $1 \leq i \leq k$.  Since $A^{\mathbb{N}}$ is Hausdorff we may define open sets $W_i'$ such that $x_i \in W_i' \subseteq W_i$ and $\{W_i' \}_{i=1}^k$ are pointwise disjoint.  Notice that $\phi|_{W_i'}$ is a homeomorphism onto its image and $\phi(W_i')$ is open.  Let $U := \cap_{i=1}^k \phi(W_i')$.  Then $U$ is an open set such that $y \in U$.  Define $V_i := \phi^{-1}(U) \cap W_i'$.  Then $x_i \in V_i$ and the $V_i$'s are open and pairwise disjoint.  Notice that $\phi^{-1}(U)= \bigsqcup_{i=1}^k V_i$.  Observe that $\phi|_{V_i}$ is a homeomorphism and $\phi(V_i)= \phi(\phi^{-1}(U) \cap W_i')= U \cap \phi(W_i')= U$.  Hence $\phi(V_i)$ is onto $U$ for each $i$.  Therefore $A^{\mathbb{N}}$ is evenly covered by $\phi$.  By Theorem~\ref{ksurj} $\phi$ is $k$-to-$1$ for some $k \in \mathbb{N}$, therefore $\phi$ is a k-fold covering map.
\end{proof}

\begin{thm} \label{last}
A sliding block code $\phi: A^{\mathbb{N}} \to A^{\mathbb{N}}$ is a local homeomorphism and $*$-commutes with the shift map $\sigma$ if and only if $\phi$ is a $k$-fold covering map defined from a regressive block map.  
\end{thm}

\begin{proof}
Since $\phi$ is a sliding block code that $*$-commutes with the shift, by Theorem~\ref{refiffstar} there exists a regressive block map $d$ such that $\tau_d = \phi$.  Since $\phi$ is a local homeomorphism that $*$-commutes with the shift, by Proposition~\ref{kcm} $\phi$ is a $k$-fold covering map.

Conversely, if $\phi$ is a $k$-fold covering map, then $\phi$ is a local homeomorphism by definition.  Since $\phi$ is defined from a regressive block map, by Theorem~\ref{refiffstar} $\phi$ $*$-commutes with the shift.
\end{proof}


\begin{thebibliography}{15}
\bibitem{AR} V. Arzumanian \& J. Renault, \emph{Examples of pseudogroups and their $C^*$-algebras},  Operator Algebras and Quantum Field Theory (Rome, 1996), 93--104, Int. Press, Cambridge, MA, 1997.

\bibitem{ER07} R. Exel \& J. Renault, \emph{Semigroups of local homeomorphisms and interaction groups}, Ergodic Theory Dynam. Systems \textbf{27} (2007), no. 6, 1737--1771.

\bibitem{Hed69} G. A. Hedlund, \emph{Endomorphisms and automorphisms of the shift dynamical systems}, Math. Systems Theory \textbf{3} (1969), 320--375.

\bibitem{jKel55} J. L. Kelley, \emph{General topology}, Reprint of the 1955 edition [Van Nostrand, Toronto, Ont.]. Graduate Texts in Mathematics, No. 27. Springer-Verlag, New York-Berlin, 1975.

\bibitem{bK98} B. Kitchens, \emph{Symbolic dynamics. One-sided, two-sided, and countable state Markov shifts}, Springer-Verlag, Berlin, 1998.

\bibitem{LM95} D. Lind \& B. Marcus, \emph{An introduction to symbolic dynamics and coding}, Cambridge University Press, Cambridge, 1995.

\bibitem{Mun} J. R. Munkres, \emph{Topology: a first course}, Prentice Hall, Inc., Englewood Cliffs, N.J., 1975.
\end{thebibliography}
\end{document}